\documentclass{amsart}
\usepackage[utf8]{inputenc}
\usepackage{a4wide}
\usepackage{amsmath,amsfonts,amssymb}
\usepackage{amsthm}
\usepackage{ifthen}
\usepackage{longtable}
\usepackage{array}
\usepackage{url}
\usepackage{hyperref}
\usepackage{enumerate}
\usepackage{enumitem}

\newcommand{\Z}{\mathbb{Z}}

\newcommand{\C}{\mathbb{C}}

\newcommand{\fdg}{\colon}
\newcommand{\eps}{\varepsilon}
\newcommand{\bb}{,\ldots ,}

\newcommand{\caselI}[1]{\medskip\noindent\textit{#1}}
\newcommand{\caselII}[1]{\smallskip\textit{#1}}

\newtheorem*{theorem*}{Theorem}
\newtheorem{thm}{Theorem}
\newtheorem{lem}{Lemma}
\newtheorem{cor}{Corollary}

\newtheorem*{claim*}{Claim}

	\newcounter{countknownthm}
	
\newtheorem{knownthm}[countknownthm]{Theorem}

\theoremstyle{definition}
	\newtheorem{defi}{Definition}
	\newtheorem{problem}{Problem}
	\newtheorem{question}{Question}
	\newtheorem{rem}{Remark}

\makeatletter
\@namedef{subjclassname@2020}{%
  \textup{2020} Mathematics Subject Classification}
\makeatother

\begin{document}
\title[Consecutive tuples of multiplicatively dependent integers]{Consecutive tuples of multiplicatively dependent integers}
\thanks{The authors were supported by the Austrian Science Fund (FWF) under the project I4406.}

\author[I. Vukusic]{Ingrid Vukusic}
\address{I. Vukusic,
University of Salzburg,
Hellbrunnerstrasse 34/I,
A-5020 Salzburg, Austria}
\email{ingrid.vukusic\char'100stud.sbg.ac.at}

\author[V. Ziegler]{Volker Ziegler}
\address{V. Ziegler,
University of Salzburg,
Hellbrunnerstrasse 34/I,
A-5020 Salzburg, Austria}
\email{volker.ziegler\char'100sbg.ac.at}

\begin{abstract}
This paper is concerned with the existence of consecutive pairs and consecutive triples of multiplicatively dependent integers.
A theorem by LeVeque on Pillai's equation implies that the only consecutive pairs of multiplicatively dependent integers larger than 1 are $(2,8)$ and $(3,9)$. 
For triples, we prove the following theorem: If $a\notin \{2,8\}$ is a fixed integer larger than 1, then there are only finitely many triples $(a,b,c)$ of pairwise distinct integers larger than 1 such that $(a,b,c)$, $(a+1,b+1,c+1)$ and $(a+2,b+2,c+2)$ are each multiplicatively dependent. Moreover, these triples can be determined effectively.
\end{abstract}

\maketitle

\section{Introduction}

We call an $n$-tuple $(z_1\bb z_n)\in \C^n$ \textit{multiplicatively dependent} if there exists a non-zero integer vector $(k_1\bb k_n)\in \Z^n$ such that 
\[
z_1^{k_1}\cdots z_n^{k_n}=1.
\]
Otherwise, we call the $n$-tuple \textit{multiplicatively independent}. 

A few years ago, Dubickas and Sha \cite{DubickasSha2018} studied multiplicative dependence of translations of algebraic numbers. They proved that given $n$ pairwise distinct algebraic numbers $(\alpha_1\bb\alpha_n)$, the translation $(\alpha_1 + t \bb \alpha_n + t)$ is multiplicatively independent if $t$ is a sufficiently large integer or a sufficiently large 	algebraic integer of sufficiently small algebraic degree.

In particular, their result implies that if $(a_1\bb a_n)$ is a multiplicatively dependent $n$-tuple of integers, then $(a_1+t \bb a_n+t)$ is multiplicatively independent, if $t$ is a sufficiently large integer.
In this paper, we ask whether this is true for small $t$, e.g. for $t=1$ (in which case we speak of consecutive $n$-tuples). We will focus on the cases $n=2$ and $n=3$, i.e. on pairs and triples.

For $n=2$ it turns out that there are in fact no consecutive pairs of multiplicatively dependent integers larger than 1 apart from $(2,8)$ and $(3,9)$. This follows quite easily from existing results \cite{LeVeque1952,Mihailescu2004} on Pillai's equation. Even stronger results for integer pairs follow from two theorems by Bennett \cite{Bennett2001,Bennett2008}, see Section \ref{sec:pairs}.

For $n=3$ there are infinitely many triples such that both $(a,b,c)$ and $(a+1,b+1,c+1)$ are multiplicatively dependent. In Section \ref{sec:triples} we characterise some of these triples and then ask the question if there are infinitely many triples such that $(a,b,c)$, $(a+1,b+1,c+1)$ and $(a+2,b+2,c+2)$ are each multiplicatively dependent. We are not able to answer this question, but we prove the following result: If $a\notin \{2,8\}$ is a fixed integer larger than 1, then there are only finitely many triples $(a,b,c)$ of pairwise distinct integers larger than 1 such that $(a,b,c)$, $(a+1,b+1,c+1)$ and $(a+2,b+2,c+2)$ are each multiplicatively dependent. This is the main theorem of this paper and Section \ref{sec:proof-allg} is devoted to its proof. The proof is based on linear forms in logarithms.

Throughout the paper we will only consider integers larger than 1, because all $n$-tuples containing 1 are multiplicatively dependent and an $n$-tuple $(a_1\bb a_n)$ is multiplicatively dependent if and only if $(|a_1|\bb |a_n|)$ is.

\section{Consecutive pairs of multiplicatively dependent integers}\label{sec:pairs}

As mentioned in the introduction, there exists exactly one pair $(a,b)$ of distinct integers larger than 1 such that $(a,b)$ and $(a+1,b+1)$ are both multiplicatively dependent, namely the pair $(2,8)$.
We will see that this follows easily from Mih\u{a}ilescu's solution of the Catalan equation~\cite{Mihailescu2004}. In fact, our statement is equivalent to a weaker result on the Catalan equation due to LeVeque \cite{LeVeque1952}. Moreover, we will use a stronger result by Bennett \cite{Bennett2008} in order to show that if $(a,b)$ is a pair of multiplicatively dependent sufficiently large distinct integers, then there are in fact no other multiplicatively dependent pairs ``in its vicinity''.

\medskip

We start by noting a useful fact about pairs of multiplicatively dependent integers: A pair of positive integers $(a,b)$ is multiplicatively dependent if and only if $a$ and $b$ are of the form $a=q^x$, $b=q^y$ for some positive integer $q$ and nonnegative integers $x,y$.

Now suppose that $1<a<b$ are integers such that $(a,b)$ and $(a+1,b+1)$ are both multiplicatively dependent. Then we can write
\begin{align}\label{eq:two}
	a&=c^x,  & b&=c^y, \\
	a+1&=d^s,  & b+1&=d^t,\nonumber
\end{align} 
with integers $c,d,x,y,s,t$. Since $1<a<b$,
we have $c>1$, $d>1$ and $y>1, t>1$. Now note that $d^t-c^y=1$ and from Mih\u{a}ilescu's solution of the Catalan equation \cite{Mihailescu2004} we get that $c^y=8$ and $d^t=9$, which implies $(a,b)=(2,8)$. 

Thus we have proven the following theorem.

\begin{thm}\label{thm:pair}
The only pair of integers $(a,b)$ with $1<a<b$ such that $(a,b)$ and $(a+1,b+1)$ are both multiplicatively dependent is $(2,8)$. 
\end{thm}

As mentioned before, Theorem \ref{thm:pair} is actually equivalent to a theorem by LeVeque \cite{LeVeque1952}.

\begin{knownthm}[LeVeque]\label{thm:leveque}
For fixed integers $d>1$, $c>1$ the equation
\begin{equation}\label{eq:leVeque}
	d^n - c^m =1
\end{equation}
has at most one integer solution $(n,m)$ except for the case $(d,c)=(3,2)$ in which it has exactly two solutions, namely $(n,m)=(1,1)$ and $(n,m)=(2,3)$.
\end{knownthm}

We check the equivalence of Theorem \ref{thm:pair} and Theorem \ref{thm:leveque}. 

First, assume that Theorem \ref{thm:leveque} is true. Suppose that $1<a<b$ are integers such that $(a,b)$ and $(a+1,b+1)$ are both multiplicatively dependent. Then from \eqref{eq:two} we get that $d^s-c^x=1$ and $d^t-c^y=1$, i.e. $(s,x)$ and $(t,y)$ are two distinct solutions of \eqref{eq:leVeque}. By Theorem \ref{thm:leveque} this is only possible if $(c,d)=(3,2)$, $(s,x)=(1,1)$ and $(t,y)=(2,3)$. This corresponds to the case $(a,b)=(2,8)$.

Now assume that Theorem \ref{thm:pair} is true. Suppose that the equation $d^n - c^m=1$ has two distinct solutions $(n,m)=(s,x)$ and $(n,m)=(t,y)$ with $x<y$. Then we can set $a=c^x$ and $b=c^y$ and we get that $(a,b)=(c^x,c^y)$ and $(a+1,b+1)=(d^s,d^t)$ are both multiplicatively dependent. By Theorem \ref{thm:pair} this means that $c^x=2$, $c^y=8$, $d^s=3$ and $d^t=9$, which implies $(d,c)=(3,2)$, $(s,x)=(1,1)$ and $(t,y)=(2,3)$.

\medskip
Next, we consider a stronger result by Bennett \cite{Bennett2008}.

\begin{knownthm}[Bennett, 2008]\label{thm:bennett}
For fixed positive integers $c,d$ the inequality
\[
	0 < \left|d^n - c^m \right| 
	< \frac{1}{4} \max \{ 	d^{n/2}, c^{m/2} \}
\] 
has at most one solution $(n,m)$ in positive integers.
\end{knownthm}

This result implies the following theorem.

\begin{thm}\label{thm:pair-vicinity}
Let $(a,b)$ be a pair of multiplicatively dependent integers with $1<a<b$. Then for any integers $\eps$, $\delta$ with $0<|\eps|< \frac{1}{4} \sqrt{a}$ and $0<|\delta|<\frac{1}{4} \sqrt{b}$ the pair $(a+\eps, b+\delta)$ is multiplicatively independent.
\end{thm}
\begin{proof}
Let $a$, $b$, $\eps$, $\delta$ be as in the theorem and suppose that  $(a,b)$ and $(a+\eps,b+\delta)$ are both multiplicatively dependent. Then we can write
\begin{align*}
	a&=c^x,  & b&=c^y, \\
	a+\eps&=d^s,  & b+\delta&=d^t,\nonumber
\end{align*} 
with positive integers $c,d,x,y,s,t$. This implies
\begin{align*}
	0 &< |\eps| 
		= |d^s - c^x | 
		< \frac{1}{4}\sqrt{a} 
		= \frac{1}{4} c^{x/2}
		\leq \frac{1}{4}\max\{d^{s/2},c^{x/2}\}
		\quad \text{and}\\
	0 &< |\delta| 
		= |d^t - c^y | 
		< \frac{1}{4}\sqrt{b} 
		= \frac{1}{4} c^{y/2}
		\leq \frac{1}{4}\max\{d^{t/2},c^{y/2}\},
\end{align*}
i.e. the inequality $0 < |d^n - c^m| < \frac{1}{4}\max\{d^{n/2},c^{s/2}\}$ has two distinct solutions $(n,m)=(s,x)$ and $(n,m)=(t,y)$, which is a contradiction to Theorem \ref{thm:bennett}.
\end{proof}

In particular, Theorem \ref{thm:pair-vicinity} implies the following.

\begin{cor}
Let $t$ be a fixed non-zero integer. Then there exist no integers $a,b$ with $16t^2 < a <b$ such that both $(a,b)$ and $(a+t,b+t)$ are multiplicatively dependent.
\end{cor}

This means, for instance, that there are no integers $64<a<b$ such that $(a,b)$ and $(a+2,b+2)$ are both multiplicatively dependent. However, we would also like to know if there are any such pairs even if $a$ is allowed to be small. 

We can obtain this from an earlier result by Bennett \cite[Theorem 5.1]{Bennett2001}, which is a generalisation of LeVeque's theorem, where the number 1 in equation \eqref{eq:leVeque} is replaced by any integer between 2 and 100.

\begin{knownthm}[Bennett, 2001]\label{thm:bennett2001}
For fixed integers $c\geq 2,d\geq 2$ and $1\leq t \leq 100$ the equation
\begin{equation}\label{eq:bennet2001}
	d^n - c^m =t
\end{equation} 
has at most one solution $(n,m)$ in positive integers, unless
\begin{align*}
	(d,c,t)\in \{
	&(3,2,1),(2,3,5),
	(2,3,13), (4,3,13),
	(16,3,13), \\
	&(2,5,3),
	(13,3,10), (91,2,89),
	(6,2,4), (15,6,9) \}.
\end{align*}
In each of these cases equation \eqref{eq:bennet2001} has precisely two solutions.
\end{knownthm}

Analogously to the equivalence of Theorem \ref{thm:leveque} and Theorem \ref{thm:pair}, Theorem \ref{thm:bennett2001} is equivalent to the following generalisation of Theorem \ref{thm:pair}.

\begin{thm}\label{thm:pair-allg}
Let $t$ be a fixed integer with $1\leq t \leq 100$ and $t\notin\{1,3,4,5,9,10,13,89\}$. Then there exist no integers $1<a<b$ such that $(a,b)$ and $(a+t,b+t)$ are both multiplicatively dependent. If $t\in \{1,3,4,5,9,10,13,89\}$, then there is exactly one pair $(a,b)$ such that $(a,b)$ and $(a+t,b+t)$ are both multiplicatively dependent.
\end{thm}

Before we move on to multiplicatively dependent triples of integers, let us note that an interesting generalisation of this section's results would be to consider algebraic numbers. We pose the following two questions.

\begin{question}\label{quest:alg-pair}
Let $K$ be a fixed number field. Are there algebraic numbers $\alpha,\beta\in K$, with $1<\alpha<\beta$ such that $(\alpha,\beta)$ and $(\alpha+1,\beta+1)$ are both multiplicatively dependent?
\end{question}

\begin{question}\label{alg-pair-vicinity}
Let $K$ be a fixed number field and let $(\alpha,\beta)\in K^2$ be a pair of multiplicatively dependent algebraic numbers of sufficiently large height. Is it possible that $(\alpha + \eps,\beta + \delta)$ is multiplicatively dependent as well, if $\eps$ and $\delta$ are algebraic numbers of small height in $K$?
\end{question}

\section{Consecutive triples of multiplicatively dependent integers}\label{sec:triples}

In this section we try to generalise the results from Section \ref{sec:pairs} to triples. Therefore, our first question is: Are there triples 
of pairwise distinct integers $(a,b,c)$ larger than 1  such that both $(a,b,c)$ and $(a+1,b+1,c+1)$ are multiplicatively dependent? 
Of course, the answer is ``yes'' as all triples with $\{2,8\}\subset\{a,b,c\}$ have this property. But even apart from these examples there are infinitely many such triples. Take for instance triples of the form $(d^x, d^y,(d^x+1)^s - 1)$ with positive integers $d,x,y,s$. Obviously $(d^x, d^y,(d^x+1)^s - 1)$ is multiplicatively dependent because $d^x$ and $d^y$ are, and $(d^x + 1, d^y+1, (d^x+1)^s)$ is multiplicatively dependent because $d^x+1$ and $(d^x+1)^s$ are. This example motivates the following definition.

\begin{defi}\label{def:k-mult}
We call an $n$-tuple $(a_1\bb a_n)$ \textit{$k$-multiplicatively dependent} if there is a multiplicatively dependent $k$-subtuple $(a_{i_1}\bb a_{i_k})$ and any $(k-1)$-subtuple $(a_{j_1}\bb a_{j_{k-1}})$ is multiplicatively independent.
\end{defi}

From the examples above we know that there are infinitely many triples $(a,b,c)$ such that $(a,b,c)$ and $(a+1,b+1,c+1)$ are both $2$-multiplicatively dependent. 

In fact, it is quite easy to characterise all such triples.

\begin{thm}\label{thm:abcd}
Let $(a,b,c)$ be a triple of pairwise distinct integers larger than 1. Then $(a,b,c)$ and $(a+1,b+1,c+1)$ are both $2$-multiplicatively dependent if and only if $(a,b,c)$ is, up to permutation, of one of the following shapes:
\begin{enumerate}[label=(\alph*)]
	\item $(2,8,d)$, with $d\notin\{2,8\}$, \label{case:a}
	\item $(8, 2^x,3^y-1)$, with $x\notin\{1,3\}$, $y>2$, \label{case:b}
	\item $(d^x, d^y, (d^x+1)^s - 1)$, with $x\neq y$, $s>1$, \label{case:c}
	\item $(d^x-1,(d^x-1)^s, d^y-1)$, with $x\neq y$, $d^x>2$, $d^y>2$, $s>1$,\label{case:d}
\end{enumerate}
with positive integers $x,y,s$ and $d>1$.
\end{thm}
\begin{proof}
It is easy to check that all triples $(a,b,c)$ of the above shapes have indeed the property that $(a,b,c)$ and $(a+1,b+1,c+1)$ are both $2$-multiplicatively dependent.

Now assume that $(a,b,c)$ is an integer triple of pairwise distinct integers larger than 1 such that $(a,b,c)$ and $(a+1,b+1,c+1)$ are both $2$-multiplicatively dependent. Without loss of generality assume that $(a,b)$ is multiplicatively dependent. If $(a+1,b+1)$ is multiplicatively dependent, then by Theorem \ref{thm:pair} we have $\{a,b\}=\{2,8\}$, which corresponds to case \ref{case:a}.
If $(a+1,b+1)$ is not multiplicatively dependent, then we may assume without loss of generality that $(a+1,c+1)$ is. Thus we can write
\begin{align*}
	a&=d^x, & b&=d^y, & c&=c,\\
	a+1 &= q^s, & b+1 &= b+1, & c+1 &=q^t,
\end{align*}
with positive integers $d>1,q>1$ and $x\neq y$, $s\neq t$.
In particular, we get that $q^s-d^x=1$. 

If $s \geq 2$ and $x\geq 2$, then by Mih\u{a}ilescu's theorem \cite{Mihailescu2004} we have $q^s=3^2$ and $d^x=2^3$, which implies $a=8$, $b=2^y$ and $c=3^t-1$, so we are in case \ref{case:b} or, if $y=1$, in case \ref{case:a}. 

If $s=1$, then we get that $c=q^t-1=(a+1)^t-1=(d^x+1)^t-1$, which corresponds to case \ref{case:c}.

If $x=1$, then we get that $b=d^y=a^y=(q^s-1)^y$. Since $a=q^s-1$ and $c=q^t-1$, this corresponds to case \ref{case:d}.
\end{proof}

The perhaps most natural next question would be: Are there any/infinitely many integer triples $(a,b,c)$ with $1<a<b<c$ such that $(a,b,c)$ and $(a+1,b+1,c+1)$ are both $3$-multiplicatively dependent? 
A brute force search reveals that such triples do indeed exist, the one with the smallest $c$ being $(9,49,63)$ (note that $(9,49,63)=(3^2, 7^2, 3^2\cdot 7)$ and $(10,50,64)=(2\cdot 5, 2\cdot 5^2,2^6)$). In the range of $2\leq a < b < c \leq 1000$ one can find 13 such triples, so there might be infinitely many, see Question \ref{quest:3-mult} in Section \ref{sec:openProblems}.

\medskip

However, there is another natural question to ask. In Section \ref{sec:pairs} we searched for consecutive multiplicatively 
dependent pairs of integers. Moving on to triples, perhaps we should consider three consecutive triples and ask whether there are any/infinitely many integer triples $(a,b,c)$ with $1<a<b<c$ such that $(a,b,c)$, $(a+1,b+1,c+1)$ and $(a+2,b+2,c+2)$ are each multiplicatively dependent. 

Again, such triples do exist and there are in fact infinitely many of them. For instance $(2,8,2^x 5^y-2)$, $(3,9,2^x 5^y-1)$ and $(4,10,2^x 5^y)$ are each multiplicatively dependent for any nonnegative integers $x$ and $y$. Excluding triples where both 2 and 8 appear, one can still find such triples, e.g. $(7,15,49)$ (note that $(7,15,49),(8,16,50)$ and $(9,17,51)$ are each multiplicatively dependent). In the range of $2\leq a < b < c \leq 1000$ there are in fact 11 such triples. The problem of deciding whether there are infinitely many of them seems to be  hard and we could not resolve it (see Question \ref{quest:3cons} in Section \ref{sec:openProblems}). However, if we simplify the problem by fixing certain integers we obtain

\begin{lem}\label{lem:zweiBeschreankt}
Let $1<a<b$ be fixed integers with $(a,b)\neq (2,8)$. Suppose that $c>1$, $c\notin\{a,b\}$ is an integer such that $(a,b,c)$, $(a+1,b+1,c+1)$, $(a+2,b+2,c+2)$ are each multiplicatively dependent. Then $c<C_0$, where $C_0=C_0(a,b)$ is an effectively computable bound depending on $a$ and $b$.
\end{lem}
\begin{proof}
By Theorem \ref{thm:pair-allg} at most one of the pairs $(a,b),(a+1,b+1),(a+2,b+2)$ can be multiplicatively dependent. Assume that the pairs $(a+i,b+i)$ and $(a+j,b+j)$ with $i,j\in \{0,1,2\}$ are each not multiplicatively dependent. Then, since $(a+i,b+i,c+i)$ is multiplicatively dependent, $c+i$ can only have prime factors that also appear in $a+i$ and/or $b+i$. Similarly, $c+j$ can only have prime factors that also appear in $a+j$ and/or $b+j$. In other words, $u:= c+i$ and $v:=c+j$ are $S$-units, where $S$ is the set of all prime factors of $(a+i)(b+i)(a+j)(b+j)$. The $S$-units satisfy the $S$-unit equation $u-v=(c+i)-(c+j)=\eps_{ij}$, where $\eps_{ij}=i-j\in \{\pm 1, \pm 2\}$. But such an $S$-unit equation has only finitely many solutions and they can be determined effectively \cite{Gyoery:1979} (see also \cite{alvaradoEtAl2020} for a practical implementation).
\end{proof}

Next, we fix only one integer. Section \ref{sec:proof-allg} is devoted to the proof of the following theorem.

\begin{thm}\label{thm:allg}
Let $a>1$ be a fixed integer with $a \notin\{2,8\}$. Suppose that $b,c \neq a$ are distinct integers larger than 1 such that $(a,b,c)$, $(a+1,b+1,c+1)$ and $(a+2,b+2,c+2)$ are each multiplicatively dependent. Then $\max\{a,b,c\}<C$, where $C=C(a)$ is an effectively computable bound depending only on $a$.
\end{thm}

\section{Proof of Theorem \ref{thm:allg}}\label{sec:proof-allg}

We will split up the proof of Theorem \ref{thm:allg} into four cases, according to how many of the triples $(a,b,c), (a+1,b+1,c+1)$ and $(a+2,b+2,c+2)$ are 2-multiplicatively dependent and how many are 3-multiplicatively dependent. We will treat each case in a separate subsection. But first, we state some auxiliary results.

\subsection{Auxiliary results.}

The main tool of the proof are lower bounds for linear forms in logarithms. For instance, one can use the bounds obtained by Matveev \cite{Matveev}. For the integer case his results imply the following bound.

\begin{knownthm}[Matveev, integer case]
Let $a_1\bb a_n \in \Z_{\geq 2}$, $b_1\bb b_n \in \Z$ and
\[
	\Lambda = b_1 \log a_1 + \dots + b_n \log a_n.
\]
Suppose that $\Lambda\neq 0$. Then we have
\[
	\log |\Lambda|
	> - 0.5 (1 + \log 2) e n^{4.5} 30^{n+3} \log a_1 \cdots \log a_n (1+ \log \max\{|b_1|\bb |b_n|\}).
\]
\end{knownthm}

If a triple $(a,b,c)$ is 2-multiplicatively dependent, then we can write two of the numbers in the form $q^x$ and $q^y$.
In order to use that fact that a specific triple is 3-multiplicatively dependent, we will need the following lemma.

\begin{lem}\label{lem:3-multdep}
Let $(a,b,c)$ be a 3-multiplicatively dependent triple of positive integers and let 
\[
a=p_1^{x_1} \cdots p_n^{x_n}
\]
be the prime factorisation of $a$.
Then $b$ and $c$ are of the form 
\[
	b=p_1^{y_1}\cdots p_n^{y_n} \cdot q^\beta, \quad
	c=p_1^{z_1}\cdots p_n^{z_n} \cdot q^\gamma,
\]
where $q,y_1\bb y_n, \beta,z_1\bb z_n,\gamma$ are nonnegative integers. Moreover, we may assume that $q$, $\beta$ and $\gamma$ are positive and that
\[
	\gamma \cdot (y_1\bb y_n) \neq \beta \cdot (z_1\bb z_n). 
\]
Also, if $q>1$, then we may assume that 
\[
	\gamma y_i \neq \beta z_i
	\quad \text{for all } i=1\bb n.
\]
\end{lem}

\begin{proof}
First, let us write $b=p_1^{y_1}\cdots p_n^{y_n} \cdot q_1$ and $c=p_1^{z_1}\cdots p_n^{z_n} \cdot q_2$ with some nonnegative integers $y_1\bb y_n,z_1\bb z_n$ and positive integers $q_1,q_2$ with $\gcd(p_i,q_1)=\gcd(p_i,q_2)=1$ for $i=1\bb n$. 

Since $(a,b,c)$ is 3-multiplicatively dependent, there are non-zero integers $r,s,t$ such that $a^rb^sc^t=1$, i.e.
\begin{equation}\label{eq:3-mult}
	(p_1^{x_1} \cdots p_n^{x_n})^r 
	(p_1^{y_1}\cdots p_n^{y_n} \cdot q_1)^s 
	(p_1^{z_1}\cdots p_n^{z_n} \cdot q_2)^t
		= 1.
\end{equation}
Since $\gcd(p_i,q_j)=1$ for $i=1\bb n$, $j=1,2$, we have in particular  $q_1^s q_2^t=1$, i.e. $q_1$ and $q_2$ are multiplicatively dependent. This means that $q_1=q^\beta$ and $q_2=q^\gamma$ for some suitable positive integer $q$ and nonnegative integers $\beta, \gamma$.

If $q>1$ and one of the exponents $\beta,\gamma$ is zero, then since $q^{\beta s}q^{\gamma t}=1$ and $r,s$ are non-zero the other exponent is zero as well. In this case we can set $q=\beta=\gamma=1$, so in any case we may assume that $\beta>0$ and $\gamma>0$.

If $\gamma \cdot (y_1\bb y_n) = \beta \cdot (z_1\bb z_n)$, then $b^\gamma c^{-\beta}=(p_1^{y_1}\cdots p_n^{y_n} \cdot q^\beta)^\gamma (p_1^{z_1}\cdots p_n^{z_n}\cdot q^\gamma)^{-\beta}=1$. This means that $(a,b,c)$ is actually 2-multiplicatively dependent and not 3-multiplicatively dependent. 

Finally, assume that $q>1$ and that $\gamma y_i = \beta z_i$ for some $i\in \{1\bb n\}$. 
Since $\gcd(p_j,q)=1$ for $j=1\bb n$ and $q>1$, we get from \eqref{eq:3-mult} that $\beta s = -\gamma t$, i.e. $t=-s\frac{\beta}{\gamma}$. Now we have
\[
	0
	= r x_i + s y_i + t z_i 
	= r x_i + s y_i - s\frac{\beta}{\gamma} z_i
	= r x_i + s y_i - s y_i
	= r x_i,
\]
which implies $x_i=0$ or $r=0$, a contradiction.
\end{proof}

Now we are ready for the proof of Theorem \ref{thm:allg}. In the next four subsections we consider one case at a time. At the end of Section \ref{sec:case0} there are some remarks on the arguments that will be reused in the further subsections.

\subsection{All triples are 2-multiplicatively dependent}\label{sec:case0}

Let $a>1$ be a fixed integer with $a \notin\{2,8\}$. Suppose that $b,c\neq a$ are distinct integers larger than 1 such that $(a,b,c)$, $(a+1,b+1,c+1)$ and $(a+2,b+2,c+2)$ are each 2-multiplicatively dependent.
We prove that there is an effectively computable bound for $b$ and $c$ which only depends on $a$.

First, note that by Theorem \ref{thm:pair-allg} at most one of the pairs $(a,b),(a+1,b+1),(a+2,b+2)$ can be multiplicatively dependent. Equally, at most one of the pairs $(a,c),(a+1,c+1),(a+2,c+2)$ can be multiplicatively dependent. Also, at most one of the pairs $(b,c),(b+1,c+1),(b+2,c+2)$ can be multiplicatively dependent, unless $\{b,c\}=\{2,8\}$. Therefore, we may assume that the pairs $(a+i,b+i)$, $(a+j,c+j)$ and $(b+k,c+k)$ are each multiplicatively dependent, where $\{i,j,k\}=\{0,1,2\}$. Thus we can write
\begin{align*}
(a+i,b+i,c+i)&=(q_i^{\alpha_i}, q_i^{\beta_i}, c+i),\\
(a+j,b+j,c+j)&=(q_j^{\alpha_j}, b+j, q_j^{\gamma_j}),\\
(a+k,b+k,c+k)&=(a+k, q_k^{\beta_k}, q_k^{\gamma_k}),
\end{align*}
where $q_i,q_j,q_k, \alpha_i,\alpha_j,\beta_i,\beta_k,\gamma_j,\gamma_k$ are positive integers and $q_i, q_j, \alpha_i, \alpha_j$ are fixed, because $a$ and therefore $a+i$ and $a+j$ are fixed integers.

We define the maximum of all unknown exponents:
\[
	M:= \max\{\beta_i, \gamma_j, \beta_k,\gamma_k\} \geq 2.
\]
The estimation $M\geq 2$ can be seen e.g. from the fact that $b+k=q_k^{\beta_k}$ and $c+k=q_k^{\gamma_k}$ are two distinct integers larger than 1, so either $\beta_k \geq 2$ or $\gamma_k \geq 2$.

Next, we can write
\begin{align}
\eps_{ik}:=i-k=(b+i)-(b+k)&=q_i^{\beta_i} - q_k^{\beta_k}, \label{eq:f0_1}\\
\eps_{jk}:=j-k=(c+j)-(c+k)&=q_j^{\gamma_j} - q_k^{\gamma_k}. \label{eq:f0_2}
\end{align}

It should be mentioned that since $q_i$ and $q_j$ are fixed, an application of a theorem by Shorey and Tijdeman \cite[Theorem 12.1]{ShoreyTijdeman1986} relatively quickly implies our desired result. However, this is a good opportunity to show how the arguments in the subsections below will work. Therefore, we directly apply linear forms in logarithms.

We divide \eqref{eq:f0_1} by $b+k=q_k^{\beta_k}$ and \eqref{eq:f0_2} by $c+k=q_k^{\gamma_k}$ and take absolute values obtaining
\begin{align}
\left|\frac{q_i^{\beta_i}}{q_k^{\beta_k}}-1\right|
	=\frac{|\eps_{ik}|}{b+k}
	\leq \frac{2}{b}, \label{eq:f0_u1}\\
\left|\frac{q_j^{\gamma_j}}{q_k^{\gamma_k}}-1\right|
	=\frac{|\eps_{jk}|}{c+k}
	\leq \frac{2}{c}. \label{eq:f0_u2}
\end{align}
If $\frac{2}{b}\geq 0.5$, then $b\leq 4$ and we are done by Lemma \ref{lem:zweiBeschreankt}. Therefore, we may assume that $\frac{2}{b} < 0.5$ and analogously that $\frac{2}{c} < 0.5$. Since $\log |x| < 2|x-1|$ for $|x-1|<0.5$, inequalities \eqref{eq:f0_u1} and \eqref{eq:f0_u2} imply
\begin{align}
	|\Lambda_1|		
		&:= |\beta_i \log q_i - \beta_k \log q_k| 
		< \frac{4}{b},\label{eq:f0_log1}\\
	|\Lambda_2|		
		&:= |\gamma_j \log q_j - \gamma_k \log q_k| 
		< \frac{4}{c}.\label{eq:f0_log2}
\end{align}
Next, we eliminate $\log q_k$ by computing
\[
	\Lambda := \gamma_k \Lambda_1 - \beta_k \Lambda_2.
\]
Thus we obtain the inequality
\begin{equation}\label{eq:f0_matv}
	|\Lambda|	
	=|\gamma_k \beta_i \log q_i - \beta_k \gamma_j q_j|
	< \frac{4 \gamma_k}{b} + \frac{4 \beta_k}{c}
	< \frac{8 M}{\min\{b,c\}}.
\end{equation}
Note that $q_i$ and $q_j$ have different prime factors, because $a+i=q_i^{\alpha_i}$ and $a+j=q_j^{\alpha_j}$ are either consecutive numbers or have difference 2. The only way they could have exactly the same prime factors would be if $a=2$, which we excluded. Therefore, $\Lambda$ can only be zero if both $\gamma_k \beta_i=0$ and $\beta_k \gamma_j=0$, which is impossible because all exponents are positive. Thus we can apply Matveev's theorem (for a better bound, at this point one might prefer to use a result by Laurent \cite{Laurent} for linear forms in two logarithms) with $n=2$, $a_1=q_i$, $b_1=\gamma_k \beta_i$, $a_2=q_j$, $b_2=-\beta_k \gamma_j$ and obtain
\[
	- C_{1} \log q_i \log q_j (1+ \log \max\{\gamma_k \beta_i, \beta_k \gamma_j\})
	< \log |\Lambda|
	< \log 8 + \log M - \log \min\{b,c\}.
\]
Since $q_i$ and $q_j$ are fixed and $\max\{\gamma_k\beta_i,\beta_k \gamma_j\} \leq M^2$, we can estimate
\[
	-C_{2} \log (M^2) < \log 8 + \log M - \log \min\{b,c\},
\]
where the constant $C_{2}$ now depends on $a$. All further constants $C_l$ in this proof are positive and only depend on $a$ as well. Shifting expressions and estimating, we obtain
\[
	\log \min\{b,c\} < C_{3} \log M.
\]
Next, note that $q_k \leq \min\{q_k^{\beta_k},q_k^{\gamma_k}\} = \min \{b+k,c+k\}\leq \min\{b,c\}+2$, so we get that
\[
	\log q_k < C_{4} \log M.
\]
Now we go back to inequality \eqref{eq:f0_log1} and apply Matveev's theorem obtaining 
\[
	-C_{5} \log q_i \log q_k (1+ \log \max\{\beta_i, \beta_k\})
	< \log | \Lambda_1| 
	< \log 4 - \log b.
\]
Since $\log q_i$ is fixed, $\log q_k < C_{4} \log M$ and $\max\{\beta_i,\beta_k\}\leq M$, we can estimate
\[
	-C_{6} (\log M)^2 < \log 4 - \log b,
\]
which implies
\begin{equation}\label{eq:f0_boundB}
	\log b < C_{7} (\log M)^2.
\end{equation}
Thus we get
\begin{align}
\max\{\beta_i, \beta_k\} 
	&= \max\{ \log (q_i^{\beta_i}) / \log q_i, \log (q_k^{\beta_k}) / \log q_k \} \label{eq:f0_boundBeta}\\
	&\leq \max\{ \log (q_i^{\beta_i}), \log (q_k^{\beta_k}) \} / \log 2 \nonumber\\
	&= \max\{\log (b+i),\log(b+k)\} / \log 2 \nonumber\\
	&\leq \log (b+2) / \log 2 \nonumber\\
	&< C_{8}(\log M)^2.\nonumber
\end{align}
Analogously, we can go back to \eqref{eq:f0_log2} and obtain
\begin{equation}\label{eq:f0_boundC}
	\log c < C_{9} (\log M)^2
\end{equation}
and
\begin{equation}\label{eq:f0_boundGamma}
	\max\{\gamma_j, \gamma_k \} < C_{10} (\log M)^2.
\end{equation}
Together, \eqref{eq:f0_boundBeta} and \eqref{eq:f0_boundGamma} give us
\[
	M < C_{11} (\log M)^2,
\]
which implies that $M$ is effectively bounded. Finally, \eqref{eq:f0_boundB} and \eqref{eq:f0_boundC} imply that $b$ and $c$ are effectively bounded as well, which completes the proof of Theorem \ref{thm:allg} in this case. 

\medskip

Before we continue, let us point out some steps that will be repeated in the following subsections in less detail.

\begin{rem}\label{rem:M}
We will always denote the maximum of all unknown exponents by $M$. 
For technical reasons (estimations involving $\log M$) we need $M\geq 2$. 
In the above proof it was obvious that the maximum of all unknown exponents was indeed at least 2. Since this will not always be the case, we will sometimes artificially include 2 in the maximum to guarantee that $M\geq 2$.
\end{rem}

\begin{rem}\label{rem:log}
We will often encounter equations similar to \eqref{eq:f0_1} of the form
\[
	\eps_{ij}
	=i-j 
	=(d+i)-(d+j)
	=q_{11}^{y_{11}}\cdots q_{1n}^{y_{1n}} - q_{21}^{z_{21}}\cdots q_{2m}^{z_{2m}},
\]
where $d=b$ or $d=c$ and $i,j$ are distinct integers in $\{0,1,2\}$. Then we may always divide by $d+j=q_{21}^{z_{21}}\cdots q_{2m}^{z_{2m}}$, take absolute values and estimate:
\[
	 \left| \frac{q_{11}^{y_{11}}\cdots p_{1n}^{y_{1n}}}{q_{21}^{z_{21}}\cdots q_{2m}^{z_{2m}}} - 1 \right|
	 \leq \frac{2}{d}.
\]
If $\frac{2}{d} \geq 0.5$, then we are done by applying Lemma \ref{lem:zweiBeschreankt}. Therefore, we may always use $\log|x|<2|x-1|$ for $|x-1|<0.5$ and obtain
\[
	| y_{11} \log q_{11} + \dots + y_{1n} \log q_{1n} 
		-  z_{21} \log q_{21} - \dots - z_{2m} \log q_{2m}|
	< \frac{4}{d}. 
\]
Also, note that such linear forms in logarithms are never equal to zero because $\eps_{ij}\neq 0$.
\end{rem}

\begin{rem}\label{rem:linform_zero}
We will often have to check that certain linear forms in logarithms, like in \eqref{eq:f0_matv}, are non-zero. The linear forms will usually be of the form
\[
	|A_{01} \log p_{01} + \dots + A_{0n} \log p_{0n} 
	+A_{11} \log p_{11} + \dots + A_{1m} \log p_{1m}
	+A_{21} \log p_{21} + \dots + A_{2k} \log p_{2k}|,
\]
where the $p$'s divide the three consecutive integers $a,a+1$ and $a+2$, i.e. $p_{01}\bb p_{0n}$ are the prime divisors of $a$, $p_{11}\bb p_{1m}$ are the prime divisors of $a+1$ and $p_{21}\bb p_{2k}$ are the prime divisors of $a+2$. Since consecutive integers have no common primes, the linear form can only vanish if all coefficients $A_{ij}$ vanish, or if $p_{01}=p_{21}=2$, $A_{01}=-A_{21}$, and all other coefficients vanish.
 We will have to check these possibilities at each occurrence.
\end{rem}

\begin{rem}\label{rem:const}
All constants $C_{l}$ will be positive, effectively computable and only depend on $a$.
\end{rem}

\begin{rem}\label{rem:matv}
We will often apply Matveev's theorem to an inequality of the form
\[
	|\Lambda|=|A_1 \log q_1 + \dots + A_n \log q_n +
	 B_1 \log r_1 + \dots + B_m \log r_m|
	< \frac{C_{l_1} M^{k_1}}{\min\{b,c\}},
\]
where the logarithms $\log q_1\bb \log q_n$ are fixed, the logarithms $\log r_1 \bb \log r_m$ are bounded by $C_{l_2} \log M$ and
all coefficients $A_i, B_j$ are bounded in absolute values by $C_{l_3} M^{k_3}$.
Then an application of Matveev's theorem yields
\[
	-C_{l_4}(\log M)^m (1+ \log (C_{l_3} M^{k_3}) )
	< \log |\Lambda|
	< \log (C_{l_1} M^{k_1}) - \log \min\{b,c\},
\]
hence we obtain
\[
	\log\min\{b,c\} 
	< C_{l_5} (\log M)^{m+1}.
\]
\end{rem}

\begin{rem}\label{rem:expos}
Estimations like in \eqref{eq:f0_boundBeta} will also be done frequently: If we have 
\[
	\log b < C_{l_1} (\log M)^m,
\]
where $b+i=p_1^{y_1}\cdots p_n^{y_n} q ^\beta$ and all integers $p_1 \cdots p_n, q$ are larger than 1, then we also have
\[
	\max\{y_1\bb y_n, \beta \} < C_{l_2} (\log M)^m.
\]
If $q=1$, then we can simply assume $\beta=1$ in the first place, so the bound still holds.
\end{rem}

\subsection{Exactly one of the triples is 3-multiplicatively dependent}\label{sec:case1}

Again, let $a>1$ be a fixed integer with $a \notin\{2,8\}$. Suppose that $b,c\neq a$ are distinct integers larger than 1 such that $(a+i,b+i,c+i)$ and $(a+j,b+j,c+j)$ are 2-multiplicatively dependent and $(a+k,b+k,c+k)$ is 3-multiplicatively dependent, where $\{i,j,k\}=\{0,1,2\}$.
We need to show that $b$ and $c$ are effectively bounded.

By Theorem \ref{thm:pair-allg} the pairs $(b+i,c+i)$ and $(b+j,c+j)$ cannot both be multiplicatively dependent, unless $\{b,c\}=\{2,8\}$. Assume without loss of generality that $(b+i,c+i)$ is not multiplicatively dependent. Since $(a+i,b+i,c+i)$ is 2-multiplicatively dependent, either $(a+i,b+i)$ or $(a+i,c+i)$ is multiplicatively dependent. Assume without loss of generality that $(a+i,b+i)$ is multiplicatively dependent. Then we are left with two cases: either $(a+j,c+j)$ is multiplicatively dependent or $(b+j,c+j)$ is multiplicatively dependent.

\caselI{Case 1:} $(a+i,b+i)$ and $(a+j,c+j)$ are each multiplicatively dependent. 
Then, using Lemma \ref{lem:3-multdep}, we can write
\begin{align*}
 (a+i,b+i,c+i) &= (q_i^{\alpha_i},q_i^{\beta_i},c+i),\\
 (a+j,b+j,c+j) &= (q_j^{\alpha_j},b+j,q_j^{\gamma_j}),\\
 (a+k,b+k,c+k) &= (p_1^{x_1}\cdots p_n^{x_n}, 
 				   p_1^{y_1}\cdots p_n^{y_n} \cdot q_k^{\beta_k},
 				   p_1^{z_1}\cdots p_n^{z_n} \cdot q_k^{\gamma_k}),
\end{align*}
where $q_i,q_j$ are fixed integers larger than 1, $p_1\bb p_n$ are fixed primes and $\alpha_i,\alpha_j, x_1\bb x_n$ are fixed positive integers. The unknown integers $y_1\bb y_n, z_1\bb z_n$ are nonnegative and the unknown integers $q_k, \beta_i,\beta_k,\gamma_j,\gamma_k$ are positive.
We set
\[
	M:=\max\{2,\beta_i, \beta_k, \gamma_j, \gamma_k, y_1 \bb y_n, z_1 \bb z_n\} \geq 2.
\]
We have
\begin{align*}
	\eps_{ik}:= i-k =(b+i)- (b+k) &= q_i^{\beta_i} - p_1^{y_1}\cdots p_n^{y_n} \cdot q_k^{\beta_k},\\
	\eps_{jk}:=j-k = (c+j)-(c+k) &= q_j^{\gamma_j} - p_1^{z_1}\cdots p_n^{z_n} \cdot q_k^{\gamma_k}.
\end{align*}
By Remark \ref{rem:log} this implies
\begin{align}
|\Lambda_{11}|:=|\beta_i \log q_i - y_1 \log p_1 - \dots - y_n \log p_n - \beta_k \log q_k| < \frac{4}{b},\label{eq:beta_i}\\
|\Lambda_{12}|:=|\gamma_j \log q_j - z_1 \log p_1 - \dots - z_n \log p_n - \gamma_k \log q_k| < \frac{4}{c}. \label{eq:gamma_j}
\end{align}
Now we eliminate $\log q_k$ by computing 
\[
	\Lambda_1:=\gamma_k\Lambda_{11}-\beta_k\Lambda_{12}.
\]
Thus we get the inequality 
\begin{multline}\label{eq:f1_matv1}
	|\Lambda_1|= 
	|\gamma_k \beta_i \log q_i - \beta_k \gamma_j \log q_j
		- (\gamma_k y_1 - \beta_k z_1) \log p_1 - \dots - (\gamma_k y_n - \beta_k z_n) \log p_n | \\
	< \frac{4 \gamma_k}{b} + \frac{4 \beta_k}{c}
	< \frac{8M}{\min\{b,c\}}.
\end{multline}
We check that $\Lambda_1 \neq 0$: 
Assume for the moment that $\Lambda_1=0$. By Remark \ref{rem:linform_zero}, this can only happen if all coefficients of all logarithms, apart from possibly $\log 2$ (or a power of 2), vanish. Since $\gamma_k \beta_i \neq 0$ and  $\beta_k \gamma_j \neq 0$, $q_i$ and $q_j$ have to be powers of 2. But then $a+i$ and $a+j$ are powers of 2, which is only possible if $a=2$, which we excluded.

Next, note that all coefficients in $\Lambda_1$ are bounded in absolute values by $M^2$ and that the logarithms $\log q_i$, $\log q_j$, $\log p_1\bb \log p_n$ are fixed.
By Remark \ref{rem:matv} an application of Matveev's theorem to \eqref{eq:f1_matv1} yields
\[
	\log \min\{b,c\}
	< C_{12} \log M.
\]
Since $q_k\leq b+k$ and $q_k \leq c+k$ we obtain
\[
	\log q_k < C_{13} \log M.
\]

Now we go back to \eqref{eq:beta_i} and apply Matveev's theorem. Since $\log q_i$ and $\log p_1 \bb \log p_n$ are fixed, $\log q_k < C_{13} \log M$ and $\beta_i, y_1 \bb y_n, \beta_k \leq M$, we obtain with Remark \ref{rem:matv}
\begin{equation}\label{eq:f1_b}
	\log b	
	< C_{14} (\log M)^2.
\end{equation}
Since $b=q_i^{\beta_i}-i=p_1^{y_1}\cdots p_n^{y_n} q_k^{\beta_k} - k$, this implies
\begin{equation}\label{eq:maxteil1}
	\max\{\beta_i,y_1 \bb y_n, \beta_k\} < C_{15}(\log M)^2.
\end{equation}

Analogously we can use inequality \eqref{eq:gamma_j} in order to obtain
\begin{equation}\label{eq:f1_c}
	\log c	
	< C_{16} (\log M)^2
\end{equation}
and
\begin{equation}\label{eq:maxteil2}
	\max\{\gamma_j, z_1 \bb z_n, \gamma_k\} < C_{17}(\log M)^2.
\end{equation}

Together, inequalities \eqref{eq:maxteil1} and \eqref{eq:maxteil2} give
\[
	M < C_{18} (\log M)^2,
\]
which implies $M<C_{19}$ and by \eqref{eq:f1_b} and \eqref{eq:f1_c} $b$ and $c$ are bounded as well.

\caselI{Case 2:} $(a+i,b+i)$ and $(b+j,c+j)$ are each multiplicatively dependent.
Then we can write
\begin{align*}
 (a+i,b+i,c+i) &= (q_i^{\alpha_i},q_i^{\beta_i},c+i),\\
 (a+j,b+j,c+j) &= (a+j,q_j^{\beta_j},q_j^{\gamma_j}),\\
 (a+k,b+k,c+k) &= (p_1^{x_1}\cdots p_n^{x_n}, 
 				   p_1^{y_1}\cdots p_n^{y_n} \cdot q_k^{\beta_k},
 				   p_1^{z_1}\cdots p_n^{z_n} \cdot q_k^{\gamma_k}),
\end{align*}
where $q_i$ is a fixed integer larger than 1, $p_1\bb p_n$ are fixed primes and $\alpha_i,x_1\bb x_n$ are fixed positive integers. The unknown integers $y_1\bb y_n, z_1\bb z_n$ are nonnegative and the unknown integers $q_j, q_k, \beta_i,\beta_k,\gamma_j,\gamma_k$ are positive.
We set
\[
	M:=\max\{\beta_i, \beta_j, \beta_k, \gamma_j, \gamma_k, y_1 \bb y_n, z_1 \bb z_n\} \geq 2.
\]
We have 
\begin{align*}
	\eps_{ik}:=i-k = (b+i)-(b+k)&= q_i^{\beta_i} - p_1^{y_1}\cdots p_n^{y_n} \cdot q_k^{\beta_k},\\
	\eps_{jk}:=j-k = (b+j)-(b+k) &= q_j^{\beta_j} - p_1^{y_1}\cdots p_n^{y_n} \cdot q_k^{\beta_k},\\
	\eps_{jk} =j-k = (c+j)-(c+k) &= q_j^{\gamma_j} - p_1^{z_1}\cdots p_n^{z_n} \cdot q_k^{\gamma_k},
\end{align*}
which implies
\begin{align}
|\Lambda_{21}|:=|\beta_i \log q_i - y_1 \log p_1 - \dots - y_n \log p_n - \beta_k \log q_k|
	&< \frac{4}{b}, \label{eq:temp1} \\
|\Lambda_{22}|:=|\beta_j \log q_j - y_1 \log p_1 - \dots - y_n \log p_n - \beta_k \log q_k|
	&< \frac{4}{b},\label{eq:temp2}\\
|\Lambda_{23}|:=|\gamma_j \log q_j - z_1 \log p_1 - \dots - z_n \log p_n - \gamma_k \log q_k|
	&< \frac{4}{c}. \label{eq:temp3}
\end{align}
First note that if $q_k=1$, then $\log q_k$ vanishes in \eqref{eq:temp1} and since $q_i$ is fixed, we can apply Matveev's theorem directly to \eqref{eq:temp1} and obtain a bound of the form 
\[
	\log \min\{b,c\}\leq \log b < C_{20} \log M.
\] 
From now on assume that $q_k>1$.

We eliminate $\log q_j$ by setting
\[
	\Lambda'_2:=\gamma_j\Lambda_{22}-\beta_j\Lambda_{23}
\]
and obtain the inequality
\begin{multline}\label{eq:temp4}
|\Lambda'_2|=| (\gamma_j y_1 - \beta_j z_1) \log p_1 + \dots
	+ (\gamma_j y_n - \beta_j z_n) \log p_n 
	+ (\gamma_j \beta_k - \beta_j\gamma_k) \log q_k |\\
< \frac{4 \gamma_j}{b} + \frac{4 \beta_j}{c}
< \frac{8M}{\min\{b,c\}}.
\end{multline}

Before we eliminate $\log q_k$ we consider the case where the term involving $\log q_k$ already vanishes in \eqref{eq:temp4}.

\caselII{Case 2.1:} The term involving $\log q_k$ vanishes in \eqref{eq:temp4}, i.e. $\gamma_j \beta_k - \beta_j\gamma_k=0$. Then \eqref{eq:temp4} becomes
\[
	|\Lambda'_{2}|=
	| (\gamma_j y_1 - \beta_j z_1) \log p_1 + \dots
	+ (\gamma_j y_n - \beta_j z_n) \log p_n |
	< \frac{8M}{\min\{b,c\}}.
\]
We check that $\Lambda'_{2}\neq 0$: Assume that $\Lambda'_{2}=0$. Then since $p_1\bb p_n$ are distinct primes, all coefficients $\gamma_j y_l - \beta_j z_l$, $l=1\bb n$, must be zero, i.e. $\gamma_j \cdot (y_1\bb y_n) = \beta_j \cdot (z_1\bb z_n)$.
By the case assumption we have $\gamma_j \beta_k = \beta_j\gamma_k$, which implies $\gamma_k \cdot (y_1\bb y_n) = \beta_k \cdot (z_1\bb z_n)$. But this is excluded in Lemma \ref{lem:3-multdep}.

Therefore $\Lambda'_{2}\neq 0$ and we can apply Matveev's theorem. Noting that $\log p_1\bb \log p_n$ are fixed and the coefficients are bounded in absolute values by $M^2$, we obtain by Remark \ref{rem:matv}
\[
	\log \min\{b,c\} 
	< C_{21} \log M.
\]

\caselII{Case 2.2:} $\gamma_j \beta_k - \beta_j\gamma_k \neq 0$.
Now we eliminate $\log q_k$ by setting 
\[
	\Lambda_2:=(\gamma_j \beta_k - \beta_j \gamma_k)\Lambda_{21}+\beta_k \Lambda'_2.
\]
Thus we get
\begin{equation}\label{eq:linform_A2}
|\Lambda_2|=
|A_0 \log q_i + A_1 \log p_1 + \dots + A_n \log p_n|
	< \frac{4|\gamma_j \beta_k - \beta_j \gamma_k|}{b} + \frac{8M\beta_k}{\min\{b,c\}}
	< \frac{12 M^2}{\min \{b,c\}},
\end{equation}
where
\begin{align*}
A_0&= \beta_i (\gamma_j \beta_k - \beta_j \gamma_k),\\
A_l&= - y_l(\gamma_j \beta_k - \beta_j \gamma_k) + \beta_k(\gamma_j y_l - \beta_j z_l)
	=\beta_j(\gamma_k y_l - \beta_k z_l), 
\quad \text{for } l=1\bb n.
\end{align*}

We check that $\Lambda_2 \neq 0$: Assume that $\Lambda_2=0$. 
By the case assumption we have $\gamma_j \beta_k - \beta_j\gamma_k \neq 0$ and since $\beta_i >0$, it follows that $A_0\neq 0$. By Remark \ref{rem:linform_zero} this means that $q_i$ is a power of 2 and that all $A_i$'s vanish except for the coefficient of $\log 2=\log p_1$. In other words, we have $A_l=0$ for all $l\geq 2$, i.e. $\gamma_k y_l = \beta_k z_l$ for $l=2\bb n$. Since $q_k>1$, Lemma \ref{lem:3-multdep} implies that $n=1$ and therefore $a+k$ is a power of 2. But we already deduced that $q_i$ is a power of 2, so $a+i$ is a power of 2 as well, which is impossible as $a\neq 2$.

Thus $\Lambda_2 \neq 0$ and we can apply Matveev's theorem to \eqref{eq:linform_A2}.
Noting that $\log q_i$, $\log p_1\bb $ $\log p_n$ are fixed and $|A_0|,|A_1|\bb |A_n|\leq M^3$, we obtain with Remark \ref{rem:matv}
\begin{equation*}
\log \min \{b,c\} < C_{22} \log M,
\end{equation*}
as we did in Case 2.1.

\smallskip

Since $q_k \leq b+k$ and $q_k \leq c+k$, we obtain
\begin{equation*}
\log q_k < C_{23}\log M.
\end{equation*}
Moreover, since $q_j \leq b+j$ and $q_j \leq c+j$, we also obtain
\[
	\log q_j
	< C_{24} \log M.
\]
Now we can go back to inequalities \eqref{eq:temp2} and \eqref{eq:temp3} and an application of Matveev's theorem to each of them yields
\begin{align}
	\log b < C_{25} (\log M)^3, \label{eq:temp6} \\
	\log c < C_{26} (\log M)^3. \label{eq:temp7}
\end{align}
Analogously to Case 1, these two inequalities imply
\[
	M < C_{27} (\log M)^3,
\]
so $M$ is effectively bounded and from \eqref{eq:temp6} and \eqref{eq:temp7} we see that $b$ and $c$ are bounded as well. 

\subsection{Exactly two triples are 3-multiplicatively dependent} \label{sec:case2}

Let $a>1$ be a fixed integer with $a \notin\{2,8\}$.
Suppose that $b,c\neq a$ are distinct integers larger than 1 such that $(a+i,b+i,c+i)$ is 2-multiplicatively dependent and $(a+j,b+j,c+j)$ and $(a+k,b+k,c+k)$ are 3-multiplicatively dependent, where $\{i,j,k\}=\{0,1,2\}$.

Since $(a+i,b+i,c+i)$ is 2-multiplicatively dependent, either $(a+i,b+i)$ or $(a+i,c+i)$ or $(b+i,c+i)$ is multiplicatively dependent.

\caselI{Case 1:} Either $(a+i,b+i)$ or $(a+i,c+i)$ is multiplicatively dependent. Assume, without loss of generality, that $(a+i,b+i)$ is multiplicatively dependent. 
Then by Lemma \ref{lem:3-multdep} we can write
\begin{align*}
(a+i,b+i,c+i)&=(q_i^{\alpha_i}, q_i^{\beta_i}, c+i),\\
(a+j,b+j,c+j)&=(p_{j1}^{x_{j1}} \cdots p_{jm}^{x_{jm}}, 
	p_{j1}^{y_{j1}} \cdots p_{jm}^{y_{jm}} \cdot q_j^{\beta_j}, 
	p_{j1}^{z_{j1}} \cdots p_{jm}^{z_{jm}} \cdot q_j^{\gamma_j}),\\
	(a+k,b+k,c+k)&=(p_{k1}^{x_{k1}} \cdots p_{kn}^{x_{kn}}, 
	p_{k1}^{y_{k1}} \cdots p_{kn}^{y_{kn}} \cdot q_k^{\beta_k}, 
	p_{k1}^{z_{k1}} \cdots p_{kn}^{z_{kn}} \cdot q_k^{\gamma_k}), 
\end{align*}
where $q_i$ is a fixed integer larger than 1, the $p$'s are fixed primes and $\alpha_i$ and the $x$'s are fixed positive integers. The $y$'s and $z$'s are unknown nonnegative integers and $q_j,q_k,\beta_1,\beta_j,\beta_k,\gamma_j$, $\gamma_k$ are also unknown and positive.
We set
\[
	M:= \max\{2,\beta_i,\beta_j,\beta_k, \gamma_j, \gamma_k, 
		y_{j1} \bb y_{jm}, z_{j1}\bb z_{jm},
		y_{k1} \bb y_{kn}, z_{k1}\bb z_{kn}\}
	\geq 2.
\]
Next we can write
\begin{align*}
\eps_{ij}:= i-j = (b+i)-(b+j) &= q_i^{\beta_i} - p_{j1}^{y_{j1}} \cdots p_{jm}^{y_{jm}} \cdot q_j^{\beta_j},\\
\eps_{ik}:= i-k = (b+i)-(b+k) &= q_i^{\beta_i} - p_{k1}^{y_{k1}} \cdots p_{kn}^{y_{kn}} \cdot q_k^{\beta_k},\\
\eps_{jk}:= j-k = (c+j)-(c+k) &= p_{j1}^{z_{j1}} \cdots p_{jm}^{z_{jm}} \cdot q_j^{\gamma_j} - p_{k1}^{z_{k1}} \cdots p_{kn}^{z_{kn}} \cdot q_k^{\gamma_k},
\end{align*}
which by Remark \ref{rem:log} implies
\begin{align}
|\Lambda_{11}|:=|\beta_i \log q_i - y_{j1} \log p_{j1} - \dots - y_{jm} \log p_{jm} - \beta_ j\log q_j|
	&< \frac{4}{b}, \label{eq:1von3}\\
|\Lambda_{12}|:=|\beta_i \log q_i - y_{k1} \log p_{k1} - \dots - y_{kn} \log p_{kn} - \beta_k \log q_k|
	&< \frac{4}{b}, \label{eq:2von3}\\
|\Lambda_{13}|:=|z_{j1} \log p_{j1} + \dots + z_{jm} \log p_{jm} + \gamma_j \log q_j  \phantom{mmmmx \frac{2}{b}}\label{eq:3von3}\\
	- z_{k1} \log p_{k1} - \dots - z_{kn} \log p_{kn} - \gamma_k \log q_k|
	 &< \frac{4}{c}.\nonumber
\end{align}
First note that if $q_j=1$ or $q_k=1$, then either $\log q_j$ vanishes in \eqref{eq:1von3} or $\log q_k$ vanishes in~\eqref{eq:2von3}. Since $q_i$ is fixed, we can apply Matveev's theorem directly to either \eqref{eq:1von3} or  \eqref{eq:2von3} and obtain
\[
	\log \min\{b,c\}
	\leq \log b
	< C_{28} \log M.
\]
From now on assume that $q_j>1$ and $q_k>1$.

We consider the linear form
\[
	\Lambda_1:=\beta_k\gamma_j\Lambda_{11}-\beta_j\gamma_k\Lambda_{12}+\beta_j\beta_k\Lambda_{13}
\]
and note that $|\beta_k\gamma_j|,|\beta_j\gamma_k|,|\beta_j\beta_k|\leq M^2$.
We obtain
\begin{multline}\label{eq:f2_1_matv}
|\Lambda_1|=
|\beta_i(\beta_k\gamma_j - \beta_j\gamma_k) \log q_i\\
	- \beta_k(\gamma_j y_{j1} - \beta_j z_{j1}) \log p_{j1} - \dots
	- \beta_k(\gamma_j y_{jm} - \beta_j z_{jm}) \log p_{jm} \\
 + \beta_j(\gamma_k y_{k1} -\beta_k z_{k1}) \log p_{k1} + \dots 
 + \beta_j(\gamma_k y_{kn} -\beta_k z_{kn}) \log p_{kn} |\\
 < \frac{4M^2}{b}+\frac{4M^2}{b}+\frac{4M^2}{c}
 < \frac{12M^2}{\min\{b,c\}}.
\end{multline}

We check that $\Lambda_1 \neq 0$:
Assume that $\Lambda_1 = 0$. By Remark \ref{rem:linform_zero} this means that either all coefficients are zero, or only $\log 2$ appears and is cancelled out.
In any case, all coefficients $\beta_k(\gamma_j y_{jl} - \beta_j z_{jl})$ for $l=2\bb m$ and $\beta_j(\gamma_k y_{kl} - \beta_k z_{kl})$ for $l=2\bb n$ are zero, i.e. $\gamma_j y_{jl} = \beta_j z_{jl}$ and $\gamma_k y_{kl} = \beta_k z_{kl}$ for $l\geq 2$. By Lemma \ref{lem:3-multdep} this means that $m=n=1$, $\gamma_k y_{j1} - \beta_k z_{j1}\neq 0$ and $\gamma_k y_{k1} - \beta_k z_{k1}\neq 0 $, i.e. $a+j$ and $a+k$ are both powers of 2, which is impossible.
	
Thus $\Lambda_1 \neq 0$
and we can apply Matveev's theorem to \eqref{eq:f2_1_matv}. Noting that all logarithms are fixed and all coefficients are bounded in absolute values by $M^3$ we obtain by Remark~\ref{rem:matv}
\[
	\log \min\{b,c\} < C_{29} \log M.
\]

Since $q_j \leq \min \{b+j,c+j\}$ and $q_k \leq \min\{b+k,c+k\}$, this implies
\[
	q_j, q_k < C_{30} \log M.
\]
As in previous subsections, we can now go back to earlier inequalities, apply Matveev's theorem two more times and obtain bounds of the form $\log b < C_{31} (\log M)^2$ and $\log c < C_{32} (\log M)^3$. Finally, this leads to an absolute upper bound for $M$ and thus for $b$ and $c$.

\caselI{Case 2:} We assume that $(b+i, c+i)$ is multiplicatively dependent.
Then we can write
\begin{align*}
(a+i,b+i,c+i)&=(a+i, q_i^{\beta_i}, q_i^{\gamma_i}),\\
(a+j,b+j,c+j)&=(p_{j1}^{x_{j1}} \cdots p_{jm}^{x_{jm}}, 
	p_{j1}^{y_{j1}} \cdots p_{jm}^{y_{jm}} \cdot q_j^{\beta_j}, 
	p_{j1}^{z_{j1}} \cdots p_{jm}^{z_{jm}} \cdot q_j^{\gamma_j}),\\
	(a+k,b+k,c+k)&=(p_{k1}^{x_{k1}} \cdots p_{kn}^{x_{kn}}, 
	p_{k1}^{y_{k1}} \cdots p_{kn}^{y_{kn}} \cdot q_k^{\beta_k}, 
	p_{k1}^{z_{k1}} \cdots p_{kn}^{z_{kn}} \cdot q_k^{\gamma_k}), 
\end{align*}
where the $p$'s are fixed primes and the $x$'s are fixed positive integers. The $y$'s and $z$'s are unknown nonnegative integers, the integers $q_j,q_k,\beta_i,\beta_j,\beta_k,\gamma_j,\gamma_k$ are also unknown and positive and $q_i$ is unknown and larger than 1.
We set
\[
	M:= \max\{\beta_i,\beta_j,\beta_k, \gamma_i, \gamma_j, \gamma_k, 
		y_{j1} \bb y_{jm}, z_{j1}\bb z_{jm},
		y_{k1} \bb y_{kn}, z_{k1}\bb z_{kn}\} \geq 2.
\]
Next we can write
\begin{align*}
\eps_{ij}:= i-j = (b+i)-(b+j) &= q_i^{\beta_i} - p_{j1}^{y_{j1}} \cdots p_{jm}^{y_{jm}} \cdot q_j^{\beta_j},\\
\eps_{ij} = i-j = (c+i)-(c+j) &= q_i^{\gamma_i} - p_{j1}^{z_{j1}} \cdots p_{jm}^{z_{jm}} \cdot q_j^{\gamma_j},\\
\eps_{ik}:= i-k = (b+i)-(b+k) &= q_i^{\beta_i} - p_{k1}^{y_{k1}} \cdots p_{kn}^{y_{kn}} \cdot q_k^{\beta_k},\\
\eps_{ik} = i-k = (c+i)-(c+k) &= q_i^{\gamma_i} - p_{k1}^{z_{k1}} \cdots p_{kn}^{z_{kn}} \cdot q_k^{\gamma_k},
\end{align*}
and by Remark \ref{rem:log} we get
\begin{align}
|\Lambda_{21}|:=|\beta_i \log q_i - y_{j1}\log p_{j1} - \dots - y_{jm}\log p_{jm} - \beta_j \log q_j | < \frac{4}{b},\label{eq:1von4}\\
|\Lambda_{22}|:=|\gamma_i \log q_i - z_{j1}\log p_{j1} - \dots - z_{jm}\log p_{jm} - \gamma_j \log q_j | < \frac{4}{c},\label{eq:2von4}\\
|\Lambda_{23}|:=|\beta_i \log q_i - y_{k1}\log p_{k1} - \dots - y_{kn}\log p_{kn} - \beta_k \log q_k | < \frac{4}{b},\label{eq:3von4}\\
|\Lambda_{24}|:=|\gamma_i \log q_i - z_{k1}\log p_{k1} - \dots - z_{kn}\log p_{kn} - \gamma_k \log q_k | < \frac{4}{c}.\label{eq:4von4}
\end{align}

Before we start eliminating $\log q_j$ and $\log q_k$ we consider the case where at least one of these two logarithms is zero in the first place.

\caselII{Case 2.1:} $q_j=1$ or $q_k=1$. Without loss of generality we may assume that $q_j=1$. Then $\log q_j = 0$ and we eliminate $\log q_i$ from \eqref{eq:1von4} and \eqref{eq:2von4}, that is we consider
\[
	\Lambda'_2:=\gamma_i\Lambda_{21}-\beta_i\Lambda_{22}.
\]
We obtain
\begin{equation}\label{eq:f2_lambda2strich}
	|\Lambda'_{2}|=	
	|(\gamma_i y_{j1} - \beta_i z_{j1})\log p_{j1} + \dots + (\gamma_i y_{jm} - \beta_i z_{jm})\log p_{jm}|
	< \frac{4\gamma_i}{b}+\frac{4\beta_i}{c}
	< \frac{8M}{\min\{b,c\}}.
\end{equation}
We check that $\Lambda'_{2}\neq 0$. Assume that $\Lambda'_{2}=0$. Since $p_{j1}\bb p_{jm}$ are distinct primes, all coefficients $\gamma_i y_{jl} - \beta_i z_{jl}$, with $l=1\bb m$, must be zero, i.e. 
\[
	\gamma_i \cdot (y_{j1}\bb y_{jm})= \beta_i \cdot (z_{j1}\bb z_{jm}).
\]
Since we assume that $q_j=1$, this implies that $(b+j)^{\gamma_i} = (c+j)^{\beta_i}$ which is a contradiction to the assumption that $(a+j,b+j,c+j)$ is 3-multiplicatively dependent.

Thus $\Lambda'_{2}\neq 0$ and we may apply Matveev's theorem to \eqref{eq:f2_lambda2strich}. Noting that all logarithms are fixed and the coefficients are bounded in absolute values by $M^2$, we obtain by Remark \ref{rem:matv}
\[
	\log \min \{b,c\} < C_{33} \log M.
\]

\caselII{Case 2.2:} $q_j>1$ and $q_k>1$. We use inequalities \eqref{eq:1von4}, \eqref{eq:2von4}, \eqref{eq:3von4} and \eqref{eq:4von4} to eliminate $\log q_j$ and $\log q_k$.
We consider the two linear forms
\begin{align*}
 \Lambda_{2a}&:=\gamma_j\Lambda_{21}-\beta_j\Lambda_{22},\\
 \Lambda_{2b}&:=\gamma_k\Lambda_{23}-\beta_k\Lambda_{24},
\end{align*}
and obtain the inequalities
\begin{align}
|\Lambda_{2a}|=|(\gamma_j\beta_i -  \beta_j\gamma_i)\log q_i 
	- (\gamma_j y_{j1}- \beta_j z_{j1})\log p_{j1} - \dots 
	- (\gamma_j y_{jm} - \beta_j z_{jm})\log p_{jm}|\label{eq:zweizusammen1}\\
< \frac{4\gamma_j}{b} + \frac{4\beta_j}{c}
< \frac{8M}{\min\{b,c\}},\nonumber\\
|\Lambda_{2b}|=|(\gamma_k\beta_i - \beta_k\gamma_i)\log q_i 
	- (\gamma_k y_{k1} - \beta_k z_{k1})\log p_{k1} - \dots 
	- (\gamma_k y_{kn} - \beta_k z_{kn})\log p_{kn}|\label{eq:zweizusammen2}\\
< \frac{4\gamma_k}{b} + \frac{4\beta_k}{c}
< \frac{8M}{\min\{b,c\}}.\nonumber
\end{align}
If either $\gamma_j\beta_i - \beta_j\gamma_i =0$ or $\gamma_k\beta_i -  \beta_k\gamma_i=0$, then $\log q_i$ vanishes in one of the inequalities and we can check that the according linear form is non-zero and apply Matveev's theorem directly to either \eqref{eq:zweizusammen1} or \eqref{eq:zweizusammen2}. Then we obtain a bound of the form
\[
	\log \min\{b,c\} < C_{34} \log M.
\]
If both $\gamma_j\beta_i - \beta_j\gamma_i \neq0$ and $\gamma_k\beta_i - \beta_k\gamma_i \neq 0$, then we eliminate $\log q_i$ from \eqref{eq:zweizusammen1} and \eqref{eq:zweizusammen2}. Therefore we set
\[
	\Lambda_2
	:=(\gamma_k\beta_i - \beta_k\gamma_i )\Lambda_{2a}
	 -(\gamma_j\beta_i - \beta_j\gamma_i )\Lambda_{2b}
\]
and note that $|\gamma_k\beta_i - \beta_k\gamma_i |,|\gamma_j\beta_i - \beta_j\gamma_i |\leq M^2$. All together we obtain
\begin{multline}\label{eq:vierZusammen}
|\Lambda_{2}|=
|A_{j1} \log p_{j1} + \dots + A_{jm} \log p_{jm} 
 + A_{k1} \log p_{k1} + \dots + A_{kn} \log p_{kn}|\\
< \frac{8M \cdot M^2}{\min\{b,c\}}
	+ \frac{8M \cdot M^2}{\min\{b,c\}}
\leq \frac{16M^3}{\min\{b,c\}},
\end{multline}
where
\begin{align*}
A_{jl}&=-(\gamma_k\beta_i   - \beta_k\gamma_i )(\gamma_j y_{jl} - \beta_j z_{jl})
\quad \text{for } l=1\bb m,\\
A_{kl}&= (\gamma_j\beta_i - \beta_j\gamma_i )(\gamma_k y_{kl} - \beta_k z_{kl})
\quad \text{for } l=1\bb n.
\end{align*}
We check that $\Lambda_{2}\neq 0$. Assume that $\Lambda_{2}=0$.
By Remark \ref{rem:linform_zero} either all coefficients are zero or only $\log 2$ appears and is cancelled out. 
In any case, all coefficients $A_{jl}$ for $l=2\bb m$ and $A_{kl}$ for $l=2\bb n$ are zero, i.e. 
$\gamma_j y_{jl} = \beta_j z_{jl}$  and $\gamma_k y_{kl} = \beta_k z_{kl}$ for $l\geq 2$. By Lemma~\ref{lem:3-multdep} this means that $m=n=1$, $\gamma_k y_{j1} - \beta_k z_{j1}\neq 0$ and $\gamma_k y_{k1} - \beta_k z_{k1}\neq 0$, i.e. $a+j$ and $a+k$ are both powers of 2, which is impossible.

Therefore $\Lambda_{2}\neq 0$ and we can apply Matveev's theorem to \eqref{eq:vierZusammen}. Noting that all logarithms are fixed and $|A_{jl}|,|A_{kl}|\leq M^4$ for all $l$,
we obtain by Remark \ref{rem:matv}
\[
	\log \min\{b,c\} < C_{35} \log M.
\]

Now, as in previous cases we can go back to earlier inequalities and prove the desired result. 

\subsection{All triples are 3-multiplicatively dependent}\label{sec:case3}

For the last time, let $a>1$ be a fixed integer with $a \notin\{2,8\}$. 
Suppose that $b,c\neq a$ are distinct integers larger than 1 such that
 $(a,b,c)$, $(a+1,b+1,c+1)$ and $(a+2,b+2,c+2)$ are each 3-multiplicatively dependent. By Lemma \ref{lem:3-multdep} we can write
\begin{align*}
(a+l,b+l,c+l) &= 
	(p_{l1}^{x_{l1}} \cdots p_{ln_l}^{x_{ln_l}}, 
	 p_{l1}^{y_{l1}} \cdots p_{ln_l}^{y_{ln_l}} \cdot q_l^{\beta_l},
	 p_{l1}^{z_{l1}} \cdots p_{ln_l}^{z_{ln_l}} \cdot q_l^{\gamma_l}),
\quad \text{for } l=0,1,2,
\end{align*}
where $n_0,n_1,n_2 \geq 1$ are fixed integers, the $p$'s are fixed primes and the $x$'s are fixed positive integers. The $y$'s and $z$'s are unknown nonnegative integers, the 
$q$'s, $\beta$'s and $\gamma$'s are unknown positive integers. 
We set
\[
	M:= \max\{2,\beta_l, \gamma_l,
		y_{l1} \bb y_{ln_l}, z_{l1}\bb z_{ln_l}
		\fdg l=0,1,2\}
	\geq 2.
\]
Assume for a moment that two of the $q$'s are equal to 1, i.e. $q_i=q_j=1$ for $i\neq j \in \{0,1,2\}$. Then both $u:=b+i$ and $v:=b+j$ are $S$-units, where $S=\{p_{i1}\bb p_{in_i},p_{j1}\bb p_{jn_j}\}$ is the set of all prime divisors of $(a+i)(a+j)$. This means that $u,v$ satisfy the $S$-unit equation $u-v=i-j=\eps_{ij}\in \{\pm 1, \pm 2\}$.
But such an $S$-unit equation has only finitely many solutions and they can be determined effectively  \cite{Gyoery:1979} (see also \cite{alvaradoEtAl2020} for a practical implementation). Analogously, $c$ is effectively bounded in this case.

Therefore, we may assume that at most one of the $q$'s is equal to 1. Let $\{i,j,k\}=\{0,1,2\}$, then we may assume that
\[
	q_j,q_k>1.
\]

We consider the equations
\begin{align*}
\eps_{ji}:=j-i=(b+j)-(b+i) &= p_{j1}^{y_{j1}} \cdots p_{jn_j}^{y_{jn_j}} \cdot q_j^{\beta_j} - p_{i1}^{y_{i1}} \cdots p_{in_i}^{y_{in_i}} \cdot q_i^{\beta_i},\\
\eps_{ji}=j-i=(c+j)-(c+i) &= p_{j1}^{z_{j1}} \cdots p_{jn_j}^{z_{jn_j}} \cdot q_j^{\gamma_j} - p_{i1}^{z_{i1}} \cdots p_{in_i}^{z_{in_i}} \cdot q_i^{\gamma_i},\\
\eps_{ki}:=k-i=(b+k)-(b+i) &= p_{k1}^{y_{k1}} \cdots p_{kn_k}^{y_{kn_k}} \cdot q_k^{\beta_k} - p_{i1}^{y_{i1}} \cdots p_{in_i}^{y_{in_i}} \cdot q_i^{\beta_i},\\
\eps_{ki}=k-i=(c+k)-(c+i) &= p_{k1}^{z_{k1}} \cdots p_{kn_k}^{z_{kn_k}} \cdot q_k^{\gamma_k} - p_{i1}^{z_{i1}} \cdots p_{in_i}^{z_{in_i}} \cdot q_i^{\gamma_i}.
\end{align*}
Let us write
\begin{align*}
\Lambda_{11}&:=y_{j1} \log p_{j1} + \dots + y_{jn_j}\log p_{jn_j} + \beta_j \log q_j 
	- y_{i1} \log p_{i1} - \dots - y_{in_i}\log p_{in_i} - \beta_i \log q_i,\\
\Lambda_{12}&:=z_{j1} \log p_{j1} + \dots + z_{jn_j}\log p_{jn_j} + \gamma_j \log q_j - z_{i1} \log p_{i1} - \dots - z_{in_i}\log p_{in_i} - \gamma_i \log q_i,\\
\Lambda_{13}&:=y_{k1} \log p_{k1} + \dots + y_{kn_k}\log p_{kn_k} + \beta_k \log q_k 
	- y_{i1} \log p_{i1} - \dots - y_{in_i}\log p_{in_i} - \beta_i \log q_i,\\
\Lambda_{14}&:=z_{k1} \log p_{k1} + \dots + z_{kn_k}\log p_{kn_k} + \gamma_k \log q_k - z_{i1} \log p_{i1} - \dots - z_{in_i}\log p_{in_i} - \gamma_i \log q_i.
\end{align*}
By Remark \ref{rem:log} this implies
\begin{align*}
|\Lambda_{11}|&< \frac{4}{b},&
|\Lambda_{12}|&< \frac{4}{c},\\
|\Lambda_{13}|&< \frac{4}{b},&
|\Lambda_{14}|&< \frac{4}{c}.
\end{align*}
We eliminate $\log q_j$ by defining
\[
	\Lambda_1:=\gamma_j\Lambda_{11}-\beta_j\Lambda_{12}
\]
and get
\begin{multline}\label{eq:q1weg}
|\Lambda_1|=|(\gamma_j y_{j1} -\beta_j z_{j1}) \log p_{j1} + \dots + (\gamma_j y_{jn_j} - \beta_j z_{jn_j})\log p_{jn_j} \\
	- (\gamma_j y_{i1} - \beta_j z_{i1}) \log p_{i1} - \dots - (\gamma_j y_{in_i} - \beta_j z_{in_i}) \log p_{in_i}
	- (\gamma_j \beta_i - \beta_j \gamma_i) \log q_i|\\
< \frac{4 \gamma_j}{b} + \frac{4 \beta_j}{c}
< \frac{8M}{\min\{b,c\}}.
\end{multline} 
Analogously we eliminate $\log q_k$ by considering
\[
	\Lambda_2:=\gamma_k\Lambda_{13}-\beta_k\Lambda_{14}
\]
and get
\begin{multline}\label{eq:q2weg}
|\Lambda_2|=|(\gamma_k y_{k1} -\beta_k z_{k1}) \log p_{k1} + \dots + (\gamma_k y_{kn_k} - \beta_k z_{kn_k})\log p_{kn_k} \\
	- (\gamma_k y_{i1} - \beta_k z_{i1}) \log p_{i1} - \dots - (\gamma_k y_{in_i} - \beta_k z_{in_i}) \log p_{in_i}
	- (\gamma_k \beta_i - \beta_k \gamma_i) \log q_i|\\
< \frac{4 \gamma_k}{b} + \frac{4 \beta_k}{c}
< \frac{8M}{\min\{b,c\}}.
\end{multline}

Before we eliminate $\log q_i$, let us consider the case where the coefficient of $\log q_i$ vanishes in either \eqref{eq:q1weg} or \eqref{eq:q2weg}.

\caselI{Case 1:} $\gamma_j \beta_i - \beta_j \gamma_i=0$ or $\gamma_k \beta_i - \beta_k \gamma_i=0$. Assume without loss of generality that $\gamma_j \beta_i - \beta_j \gamma_i=0$. Then \eqref{eq:q1weg} becomes
\begin{multline}\label{eq:f3_lambda1}
|\Lambda_1|=
|(\gamma_j y_{j1} -\beta_j z_{j1}) \log p_{j1} + \dots + (\gamma_j y_{jn_j} - \beta_j z_{jn_j})\log p_{jn_j} \\
	- (\gamma_j y_{i1} - \beta_j z_{i1}) \log p_{i1} - \dots - (\gamma_j y_{in_i} - \beta_j z_{in_i}) \log p_{in_i}|
< \frac{8M}{\min\{b,c\}}.
\end{multline} 

We check that $\Lambda_1\neq 0$:
Assume that $\Lambda_1 = 0$. By Remark \ref{rem:linform_zero} this means that either all coefficients are zero, or only $\log 2$ appears and is cancelled out.
In any case, all coefficients 
$\gamma_j y_{jl} -\beta_j z_{jl}$ and $\gamma_j y_{il} - \beta_j z_{il}$ are zero for $l\geq 2$.

On the one hand, we have  $\gamma_j y_{jl} = \beta_j z_{jl}$ for $l=2\bb n_j$.
Since $q_j>1$, Lemma \ref{lem:3-multdep} implies that $n_j=1$ and $\gamma_j y_{j1} \neq \beta_j z_{j1}$. Then we must have $\log p_{j1}=\log p_{i1}=\log 2$ and 
\begin{equation}\label{eq:f3_extra}
	\gamma_j y_{j1} -\beta_j z_{j1}
	=\gamma_j y_{i1} - \beta_j z_{i1}.
\end{equation} 
In particular, note that $a+j$ is a power of 2 and that $a+j$ and $a+i$ are both even, so $|i-j|=2$.

On the other hand, we have $\gamma_j y_{il} = \beta_j z_{il}$ for $l=2\bb n_i$. By the case assumption we have $\gamma_j \beta_i = \beta_j \gamma_i$, so we get $\gamma_i y_{il} = \beta_i z_{il}$ for $l=2\bb n_i$. 

If $q_i>1$, then 
Lemma \ref{lem:3-multdep} implies that $n_i=1$ and therefore $a+i$ is a power of 2, which is a contradiction because $a+i$ and $a+j$ cannot both be powers of 2.

If $q_i=1$, then we may assume that $\beta_i=\gamma_i=1$. From the case assumption $\gamma_j \beta_i - \beta_j \gamma_i=0$ we get that $\beta_j=\gamma_j$.
Then \eqref{eq:f3_extra} becomes
\begin{equation}\label{eq:f3_extra2}
	y_{j1} - z_{j1}
	=y_{i1} - z_{i1}.
\end{equation}
Since $b+j=2^{y_{j1}}q_j^{\beta_j}$ and $c+j=2^{z_{j1}}q_j^{\gamma_j}=2^{z_{j1}}q_j^{\beta_j}$, we have $y_{j1}\neq z_{j1}$ (otherwise $(a+j,b+j,c+j)$ would be 2-multiplicatively dependent). Without loss of generality assume that $y_{j1}<z_{j1}$. In order to get a contradiction, we check the three cases $y_{j1}=0$, $y_{j1}=1$ and $y_{j1}\geq 2$:
\begin{itemize}
\item $y_{j1}=0$ and $z_{j1}\geq 1$. Then since $|i-j|=2$, the integers $b+j$ and $b+i$ must both be odd, so we have $y_{i1}=0$ as well. Then equation \eqref{eq:f3_extra2} implies $z_{j1}=z_{i1}$, which is impossible because $c+j$ and $c+i$ have difference 2 so they cannot have the same positive 2-adic valuation.
\item $y_{j1}=1$ and $z_{j1}\geq 2$. Since $|i-j|=2$, this implies $y_{i1}\geq 2$ and $z_{i1}=1$. Then $y_{j1}-z_{j1}<0$ and $y_{i1}-z_{i1}>0$, a contradiction to~\eqref{eq:f3_extra2}.
\item $y_{j1}\geq 2$ and $z_{j1}> y_{j1} \geq 2$. Then since $|i-j|=2$, we have $y_{i1}=z_{i1}=1$, so $y_{i1}-z_{i1}=0$ and $y_{j1}-z_{j1}<0$, also a contradiction to~\eqref{eq:f3_extra2}.
\end{itemize}

Thus we may finally assume that $\Lambda_1\neq 0$ and apply Matveev's theorem to \eqref{eq:f3_lambda1}. Noting that all logarithms are fixed and all coefficients are bounded in absolute values by $M^2$, we obtain by Remark~\ref{rem:matv}
\[
	\log \min \{b,c\} < C_{36} \log M.
\]

\caselI{Case 2:} $\gamma_j \beta_i - \beta_j \gamma_i \neq 0$ and $\gamma_k \beta_i - \beta_k \gamma_i \neq 0$.
We eliminate $\log q_i$ and consider
\[
	\Lambda:= (\gamma_k \beta_i - \beta_k \gamma_i)\Lambda_1-(\gamma_j \beta_i - \beta_j \gamma_i)\Lambda_2.
\]
Noting that $|\gamma_k \beta_i - \beta_k \gamma_i|,|\gamma_j \beta_i - \beta_j \gamma_i|\leq M^2$, we get the inequality
\begin{multline}\label{eq:q0auchweg}
|\Lambda|=
| A_{i1} \log p_{i1} + \dots + A_{in_i} \log p_{in_i} +
	A_{j1} \log p_{j1} + \dots + A_{jn_j} \log p_{jn_j} \\
	+ A_{k1} \log p_{k1} + \dots + A_{kn_k} \log p_{kn_k}|
< \frac{8M\cdot M^2}{\min\{b,c\}}
	+ \frac{8M\cdot M^2}{\min\{b,c\}}
\leq \frac{16M^3}{\min\{b,c\}},
\end{multline}
where
\begin{align*}
A_{il}&= -(\gamma_k \beta_i - \beta_k \gamma_i)(\gamma_j y_{il} - \beta_j z_{il})
	+ (\gamma_j \beta_i - \beta_j \gamma_i)(\gamma_k y_{il} - \beta_k z_{il})\\
	&=(\gamma_j\beta_k - \beta_j \gamma_k)(\gamma_i y_{il} - \beta_i z_{il})
	\quad \text{for } l=1\bb n_i,\\
A_{jl}&= (\gamma_k \beta_i - \beta_k \gamma_i)(\gamma_j y_{jl} -\beta_j z_{jl})	
	\quad \text{for } l=1\bb n_j,\\
A_{kl}&= -(\gamma_j \beta_i - \beta_j \gamma_i)(\gamma_k y_{kl} -\beta_k z_{kl})
	\quad \text{for } l=1\bb n_k.
\end{align*}
We check that $\Lambda\neq 0$: 
Assume that $\Lambda=0$. 
By Remark \ref{rem:linform_zero} this means that either all coefficients are zero, or only $\log 2$ appears and is cancelled out.
In any case, all coefficients $A_{il},A_{jl},A_{kl}$ are zero for $l\geq 2$. 
With the case assumption we get that in particular $\gamma_j y_{jl} -\beta_j z_{jl}=0$ for $l=2\bb n_j$ and $\gamma_k y_{kl} -\beta_k z_{kl}=0$ for $l=2\bb n_k$. Since we assumed that $q_j,q_k>1$, Lemma \ref{lem:3-multdep} implies that $n_j=n_k=1$, $\gamma_j y_{jl} -\beta_j z_{j1}\neq 0$, $\gamma_k y_{k1} -\beta_k z_{k1}\neq 0$ and $a+j, a+k$ are both powers of 2, which is impossible.

Thus $\Lambda\neq 0$ and we can apply Matveev's theorem to \eqref{eq:q0auchweg}. Noting that 
the $A$'s are bounded in absolute values by $M^4$, we obtain with Remark~\ref{rem:matv}
\[
	\log \min\{b,c\} < C_{37} \log M.
\]
Then we can go back to earlier inequalities and finish the proof as in the other cases. 

\section{Further open problems}\label{sec:openProblems}

The authors did the computations from Section \ref{sec:case3} for the case $a=3$, i.e. they searched for triples of the form $(3,b,c)$ such that $(3,b,c)$, $(4,b+1,c+1)$, $(5,b+2,c+2)$ are each 3-multiplicatively dependent. The resulting bound was
\[
	\max\{b,c\} < 10^{2.4\cdot 10^{58}}.
\]
In fact, with the notation $(a+i,b+i,c+i)=(p_i^{x_i}, p_i^{y_i}q_i^{\beta_i}, p_i^{z_i}q_i^{\gamma_i})$ for $i=0,1,2$ and $p_0=3,p_1=2,p_2=5$ the bounds $\max\{q_1,q_2,q_3\} < 2\cdot 10^{14}$ and $\max\{\beta_i, \gamma_i, y_i, z_i\fdg i=0,1,2\} < 2.7 \cdot 10^{44}$ were obtained. However, it does not seem possible to apply a reduction method (e.g. the LLL-algorithm) because the bound for  $\max\{q_1,q_2,q_3\}$ is way too large. 
Probably similar bounds and difficulties are encountered in the other cases (if some of the triples are 2-multiplicatively dependent).
Therefore, we pose the following problem.

\begin{problem}\label{probl:a3}
Find all distinct integers $b,c$ larger than 1 and different from 3, such that $(3,b,c)$, $(4,b+1,c+1)$ and $(5,b+2,c+2)$ are each multiplicatively dependent.
\end{problem}
In the range $2\leq b < c \leq 10^5$ a brute force search revealed just one such triple, namely $(3,b,c)=(3,2,8)$. It seems plausible that this might be the only solution to Problem  \ref{probl:a3}.

\medskip

Now we go back to the original question in Section \ref{sec:triples}, where $a$ is not fixed.

\begin{question}\label{quest:3cons}
Are there infinitely many triples $(a,b,c)$ of pairwise distinct integers lager than 1 with $\{2,8\}\not\subset \{a,b,c\}$ such that $(a,b,c)$, $(a+1,b+1,c+1)$ and $(a+2,b+2,c+2)$ are each multiplicatively dependent?
\end{question}

A brute force search revealed that in the range $2\leq a<b<c\leq 10^3$ there are 11 such triples. In 8 occurrences exactly one of the triples $(a,b,c),(a+1,b+1,c+1),(a+2,b+2,c+2)$ is 3-multiplicatively dependent:
\begin{align*}
(2,	4,	14),
(2,	6,	48),
(2,	14,	224),
(2,	30,	960),
(6,	30,	216),
(7,	15,	49),
(7,	49,	79),
(8,	32,	98).
\end{align*}
In the following 3 occurrences two of the triples are 3-multiplicatively dependent:
\[
(3,	6,	48),
(6,	8,	48),
(6,	18,	48).
\]
In the range of the search there were no occurrences where all triples were 2-multiplicatively dependent or where all triples were 3-multiplicatively dependent. Of course, the range was too small to allow any well-founded conjecture, but it seems possible that there are infinitely many occurrences of three consecutive multiplicatively dependent integers. However, there might be fewer if we allow only 2-multiplicatively dependent or only 3-multiplicatively dependent triples. 
From Theorem \ref{thm:abcd} we know the shape of all occurrences of two consecutive 2-multiplicatively dependent triples. Requiring that the third consecutive triple be 2-multiplicatively dependent as well leads to special types of Pillai equations,  some of which are easy to solve. However, the authors have not been able to solve the following equations in integers larger than 1 and $x\neq y$:
\begin{align}
\left( (d^x + 1)^s + 1\right) ^t - d^y &= 2,\label{eq:op1}\\
d^y - \left( (d^x - 1)^s - 1\right) ^t &=2,\label{eq:op2}\\
(d^x+1)^s - (d^y-1)^m &= 2.\label{eq:op3}
\end{align}

\begin{problem}
Solve equations \eqref{eq:op1}, \eqref{eq:op2} and \eqref{eq:op3} and thus determine the number of triples $(a,b,c)$ of pairwise distinct integers lager than 1 such that $(a,b,c)$, $(a+1,b+1,c+1)$ and $(a+2,b+2,c+2)$ are each 2-multiplicatively dependent.
\end{problem}

\medskip

If we only consider triples that are 3-multiplicatively dependent, we can go back to the following question from Section \ref{sec:triples}.

\begin{question}\label{quest:3-mult}
Are there infinitely many triples $(a,b,c)$ of pairwise distinct integers lager than 1 such that $(a,b,c)$ and $(a+1,b+1,c+1)$ are both 3-multiplicatively dependent?
\end{question}

In the range $2\leq a < b < c \leq 10^3$ a brute force search revealed 13 such triples.


\medskip
 
If any of these questions can be answered, it would be interesting to generalise the results to consecutive $n$-tuples.

\begin{question}
Are there infinitely many $n$-tuples $(a_1\bb a_n)$ of pairwise distinct integers larger than 1 such that $(a_1\bb a_n),(a_1+1 \bb a_n +1)\bb (a_1+(n-1)\bb a_n +(n-1))$ are each multiplicatively dependent? Are there infinitely many $n$-tuples such that $(a_1\bb a_n)$ and $(a_1 +1 \bb a_n+1)$ are both $n$-multiplicatively dependent?
\end{question}

\bibliographystyle{plain}
\bibliography{lit_multAbh}

\begin{thebibliography}{10}

\bibitem{alvaradoEtAl2020}
A.~Alvarado, A.~Koutsianas, B.~Malmskog, C.~Rasmussen, C.~Vincent, and M.~West.
\newblock A robust implementation for solving the $s$-unit equation and several
  applications, 2020.
\newblock \href{https://arxiv.org/abs/1903.00977}{arXiv:1903.00977}.

\bibitem{Bennett2001}
M.~A. Bennett.
\newblock On some exponential equations of {S}. {S}. {P}illai.
\newblock {\em Canad. J. Math.}, 53(5):897--922, 2001.

\bibitem{Bennett2008}
M.~A. Bennett.
\newblock Differences between perfect powers.
\newblock {\em Canad. Math. Bull.}, 51(3):337--347, 2008.

\bibitem{DubickasSha2018}
A.~Dubickas and M.~Sha.
\newblock Multiplicative dependence of the translations of algebraic numbers.
\newblock {\em Rev. Mat. Iberoam.}, 34(4):1789--1808, 2018.

\bibitem{Gyoery:1979}
K.~{Gy\H{o}ry}.
\newblock {On the number of solutions of linear equations in units of an
  algebraic number field}.
\newblock {\em {Comment. Math. Helv.}}, 54:583--600, 1979.

\bibitem{Laurent}
M.~Laurent.
\newblock Linear forms in two logarithms and interpolation determinants ii.
\newblock {\em Acta Arith.}, 133(4):325--348, 2008.

\bibitem{LeVeque1952}
W.~J. LeVeque.
\newblock On the equation {$a^x-b^y=1$}.
\newblock {\em Amer. J. Math.}, 74:325--331, 1952.

\bibitem{Matveev}
E.~M. Matveev.
\newblock An explicit lower bound for a homogeneous rational linear form in the
  logarithms of algebraic numbers. {II}.
\newblock {\em Izv. Math.}, 64(6):1217--1269, 2000.

\bibitem{Mihailescu2004}
P.~Mih\u{a}ilescu.
\newblock Primary cyclotomic units and a proof of {C}atalan's conjecture.
\newblock {\em J. f{\"u}r die Reine und Angew. Math.}, pages 167--195, 2004.

\bibitem{ShoreyTijdeman1986}
T.~N. Shorey and R.~Tijdeman.
\newblock {\em Exponential diophantine equations}.
\newblock Cambridge University Press, 1986.

\end{thebibliography}

\end{document}